\documentclass[a4paper,reqno,11pt]{amsart}
\usepackage{amsmath}
\usepackage{amssymb}
\usepackage{cases}
\usepackage{enumitem}
\usepackage{array, longtable}

\allowdisplaybreaks[4]

\newtheorem{thm}{Theorem}[section]
\newtheorem{prop}[thm]{Proposition}

\newtheorem{lem}[thm]{Lemma}
\newtheorem{q}[thm]{Question}

\newtheorem{cor}[thm]{Corollary}

\newtheorem{claim}[thm]{Claim}

\theoremstyle{definition}
\newtheorem{definition}[thm]{Definition}

\newtheorem{set}[thm]{Setup}

\theoremstyle{remark}

\numberwithin{equation}{section}

\newcommand{\bQ}{\mathbb{Q}}

\newcommand\OO{{\mathcal{O}}}

\newcommand{\Proj}{\operatorname{Proj}}

\usepackage{todonotes}

\begin{document}

\title{Characterizing terminal Fano threefolds with the smallest anti-canonical volume, II}
\date{\today}
\author{Chen Jiang}
\address{Chen Jiang, Shanghai Center for Mathematical Sciences, Fudan University, Jiangwan Campus, Shanghai, 200438, China}
\email{chenjiang@fudan.edu.cn}

\dedicatory{Dedicated to Professor Vyacheslav Shokurov on the occasion of his 70th birthday}

\begin{abstract}
It was proved by J.~A.~Chen and M.~Chen that a terminal Fano $3$-fold $X$ satisfies $(-K_X)^3\geq \frac{1}{330}$.
We show that a $\bQ$-factorial terminal Fano $3$-fold $X$ with $\rho(X)=1$ and $(-K_X)^3=\frac{1}{330}$ is a weighted hypersurface of degree $66$ in $\mathbb{P}(1,5,6,22,33)$.

By the same method, we also give characterizations for other $11$ examples of weighted hypersurfaces of the form $X_{6d}\subset \mathbb{P}(1,a,b,2d,3d)$ in Iano-Fletcher's list. Namely, we show that if a $\bQ$-factorial terminal Fano $3$-fold $X$ with $\rho(X)=1$ has the same numerical data as $X_{6d}$, then $X$ itself is a weighted hypersurface of the same type.
\end{abstract}

\maketitle
\pagestyle{myheadings} \markboth{\hfill C.~Jiang \hfill}{\hfill Terminal Fano $3$-folds with the smallest anti-canonical volume, II\hfill}


\section{Introduction}
Throughout this note, we work over the field of complex numbers $\mathbb{C}$.

A normal projective variety $X$ is called a {\it Fano variety} (resp., {\it weak Fano variety}) if the anti-canonical divisor $-K_X$ is ample (resp., nef and big). 
A {terminal} (resp. canonical) Fano variety is a Fano variety with at worst terminal (resp., canonical) singularities. A {terminal} (resp. canonical) weak Fano variety is a weak Fano variety with at worst terminal (resp., canonical) singularities.

According to the minimal model program, Fano varieties form a fundamental class among research objects of birational geometry. Motivated by the classification theory of $3$-dimensional algebraic varieties, we are interested in the study of explicit geometry of terminal/canonical Fano $3$-folds.

Given a terminal weak Fano $3$-fold $X$, it was proved in \cite[Theorem~1.1]{CC08} that $(-K_{X})^3\geq \frac{1}{330}$. This lower bound is optimal, as it is attained when $X=X_{66}\subset \mathbb{P}(1,5,6,22,33)$ is a general weighted hypersurface of degree $66$. Moreover, \cite[Theorem~1.1]{CC08} showed that when $(-K_{X})^3= \frac{1}{330}$, then $X$ has exactly the same Reid basket of virtual orbifold singularities as $X_{66}$.
So it is interesting to ask the following question:
\begin{q}\label{main q}
Let $X$ be a terminal (weak) Fano $3$-fold with $(-K_X)^3=\frac{1}{330}$. 
Is $X$ a $\mathbb{Q}$-Gorenstein deformation of a quasi-smooth weighted hypersurface of degree $66$ in 
 $\mathbb{P}(1,5,6,22,33)$?
 \end{q}

In \cite{330}, we give a partial answer to this question in the category of non-rational $\bQ$-factorial terminal Fano $3$-folds with $\rho=1$. The main purpose of this note is to remove the non-rationality condition. The following is the main theorem of this note.
\begin{thm}\label{mainthm}
Let $X$ be a $\bQ$-factorial terminal Fano $3$-fold with $\rho(X)=1$ and $(-K_X)^3=\frac{1}{330}$.
Then $X$ is a weighted hypersurface in $\mathbb{P}(1,5,6,22,33)$ defined by a weighted homogeneous polynomial $F$ of degree $66$, where 
$$F(x, y, z, w, t)=t^2+F_0(x, y, z, w)$$ in suitable homogeneous coordinates $[x: y: z: w: t]$ of $\mathbb{P}(1, 5, 6, 22, 33)$.
\end{thm}

As mentioned in \cite[Remark~1.3]{330}, similar method can be applied to characterize other weighted hypersurfaces in Iano-Fletcher's list \cite[16.6]{IF00}. We list them in Table~\ref{tableA}. The weighted hypersurfaces in Table~\ref{tableA} are all of the form $X_{6d}\subset \mathbb{P}(1,a,b,2d,3d)$ with $a+b=d$ and any such general hypersurface is a $\bQ$-factorial terminal Fano $3$-fold with $\rho=1$. The geometry of them are very interesting in the sense that the behavior of the anti-pluri-canonical systems of them match perfectly with the theoretical estimates in our prediction (\cite[Example~5.12]{CJ16}, \cite[Example~4.4.12]{Phd}).

\begin{longtable}{|l|l|l|l|}
\caption{Fano $3$-folds from Iano-Fletcher's list}\label{tableA}\\
		\hline
			No. & hypersurface & $-K^3$ & basket \\ \hline	
\endfirsthead
\multicolumn{4}{l}{{ {\bf \tablename\ \thetable{}} \textrm{-- continued from previous page}}} \\
\hline 
			No. & hypersurface & $-K^3$ & basket \\ \hline	
\endhead

 \multicolumn{4}{l}{{\textrm{Continued on next page}}} \\ \hline
\endfoot
\hline 
\endlastfoot

14 & ${X_{12}}\subset \mathbb{P}(1,1,1,4,6)$ & $1/2$ & $ (1,2) $\\ \hline
34 & ${X_{18}}\subset \mathbb{P}(1,1,2,6,9)$ & $1/6$ &$ 3\times(1,2),(1,3) $\\ \hline
53 & ${X_{24}}\subset \mathbb{P}(1,1,3,8,12)$ & $1/12$ &$ 2\times(1,3), (1,4) $\\ \hline
70 & ${X_{30}}\subset \mathbb{P}(1,1,4,10,15)$ & $1/20$ &$ (1,2), (1,4), (1,5) $\\ \hline
72 & ${X_{30}}\subset \mathbb{P}(1,2,3,10,15)$ & $1/30$ &$ 3\times(1,2), (2, 5), 2\times(1,3)$\\ \hline
82 & ${X_{36}}\subset \mathbb{P}(1,1,5,12,18)$ & $1/30$ &$ (2,5), (1,6) $\\ \hline
88 & ${X_{42}}\subset \mathbb{P}(1,1,6,14,21)$ & $1/42$ &$ (1,2), (1,3), (1,7)$\\ \hline
89 & ${X_{42}}\subset \mathbb{P}(1,2,5,14,21)$ & $1/70$ &$ 3\times(1,2),(3,7), (1,5) $\\ \hline
90 & ${X_{42}}\subset \mathbb{P}(1,3,4,14,21)$ & $1/84$ &$ (1,2),2\times(1,3), (2,7), (1,4) $\\ \hline
92 & ${X_{48}}\subset \mathbb{P}(1,3,5,16,24)$ & $1/120$ &$ (3,8),2\times(1,3), (1,5) $\\ \hline
94 & ${X_{54}}\subset \mathbb{P}(1,4,5,18,27)$ & $1/180$ &$ (1,2), (2,5), (1,4), (2,9) $\\ \hline
95 & ${X_{66}}\subset \mathbb{P}(1,5,6,22,33)$ & $1/330$ &$ (1,2), (2,5), (1,3), (2,11) $
\end{longtable}

We show the following result saying that a $\bQ$-factorial terminal Fano $3$-fold with $\rho=1$ shares the same numerical information with a weighted hypersurface in Table~\ref{tableA} is always a weighted hypersurface of the same type. 

\begin{thm}\label{mainthm2}
Let $X$ be a $\bQ$-factorial terminal Fano $3$-fold with $\rho(X)=1$ such that $(-K_X)^3=(-K_{X_{6d}})^3$ and $B_{X}=B_{X_{6d}}$ for some $$X_{6d}\subset \mathbb{P}(1,a,b,2d, 3d)$$ as in Table~\ref{tableA}. 
Then $X$ is a weighted hypersurface in $\mathbb{P}(1,a,b,2d,3d)$ defined by a weighted homogeneous polynomial $F$ of degree $6d$, where 
$$F(x, y, z, w, t)=t^2+F_0(x, y, z, w)$$ in suitable homogeneous coordinates $[x: y: z: w: t]$ of $\mathbb{P}(1, a, b, 2d, 3d)$.
\end{thm}

According to \cite{330}, the idea of the proof is as the following: we construct a rational map $\Phi_{3d}: X\dashrightarrow \mathbb{P}(1, a, b, 2d, 3d)$ by general global sections of $H^0(X, -mK_X)$ and show that $\Phi_{3d}$ is indeed an embedding. The reason that we need to assume non-rationality of $X$ in \cite{330} is that we need to rule out the case that the induced map 
$X\dashrightarrow \mathbb{P}(1, a, b, 2d)$ is birational. To remove this condition, one idea is to directly show that $X$ is non-rational under this setting, which seems very difficult (as non-rationality of Fano varieties is always a difficult problem). Our new ingredient is to apply the theory in \cite{CJ16, Phd} to show directly that $X\dashrightarrow \mathbb{P}(1, a, b, 2d)$ cannot be birational. This involves detail discussions on the behavior of anti-pluri-canonical systems of $X$. 
We will discuss another approach (unsuccessful yet) in the end which is related to Shokurov's earlier work.

\section{Reid's Riemann--Roch formula}\label{sec 2}

A {\it basket} $B$ is a collection of pairs of integers (permitting
weights), say $\{(b_i,r_i)\mid i=1, \cdots, s; b_i\ \text{is coprime
 to}\ r_i\}$. 
 
Let $X$ be a canonical weak Fano $3$-fold and $Y$ be a terminalization of $X$. According to 
Reid \cite{YPG}, there is a basket of orbifold points (called {\it Reid basket})
$$B_X=\bigg\{(b_i,r_i)\mid i=1,\cdots, s; 0<b_i\leq \frac{r_i}{2};b_i \text{ is coprime to } r_i\bigg\}$$
associated to $X$, which comes from locally deforming singularities of $Y$ into cyclic quotient singularities, where a pair $(b_i,r_i)$ corresponds to a (virtual) orbifold point $Q_i$ of type $\frac{1}{r_i}(1,-1,b_i)$. 

Recall that for a Weil divisor $D$ on $X$, $$H^0(X, D)=\{f\in \mathbb{C}(X)^{\times}\mid \text{div}(f)+D\geq 0\}\cup \{0\}.$$
By Reid's Riemann--Roch formula and the Kawamata--Viehweg vanishing theorem, for any positive integer $m$, 
\begin{align*}
h^0(X, -mK_X)={}&\chi(X, \OO_X(-mK_X))\\={}&\frac{1}{12}m(m+1)(2m+1)(-K_X)^3+(2m+1)-l(m+1)
\end{align*}
where
$l(m+1)=\sum_i\sum_{j=1}^m\frac{\overline{jb_i}(r_i-\overline{jb_i})}{2r_i}$ and the first sum runs over all orbifold points in Reid basket (\cite[2.2]{CJ16}). Here $\overline{jb_i}$ means the smallest non-negative residue of $jb_i \bmod r_i$.

We will freely use the following lemma to compute anti-pluri-genera.
\begin{lem}\label{lem Xd RR}
Let $X$ be a canonical weak Fano $3$-fold such that $(-K_X)^3=(-K_{X_{d}})^3$ and $B_{X}=B_{X_{d}}$ for some weighted hypersurface
$$X_{d}\subset \mathbb{P}(a_0,a_1,a_2,a_3, a_4)$$ 
as in Iano-Fletcher's list \cite[16.6]{IF00}. Then
$$
\sum_{m\geq 0}h^0(X, -mK_X)q^m=\frac{1-q^{d}}{(1-q^{a_0})(1-q^{a_1})(1-q^{a_2})(1-q^{a_3})(1-q^{a_4})}.
$$
 \end{lem}
\begin{proof}
By Reid's Riemann--Roch formula, $h^0(X, -mK_X)$ depends only on $(-K_X)^3$ and $B_X$. Note that $\mathcal{O}_{{X_{d}}}(-K_{{X_{d}}})=\mathcal{O}_{\mathbb{P}}(1)|_{{X_{d}}}$ where $\mathcal{O}_{\mathbb{P}}(1)$ is the natural twisting sheaf of $\mathbb{P}(a_0,a_1,a_2,a_3, a_4)$.
So 
\begin{align*}
\sum_{m\geq 0}h^0(X, -mK_X)q^m={}&\sum_{m\geq 0}h^0({X_{d}}, \mathcal{O}_{\mathbb{P}}(m)|_{{X_{d}}} )q^m,
\end{align*}
and the last series can be computed by
by \cite[Theorem~3.4.4]{Dol82}.
\end{proof}

 \section{Generic finiteness and birationality}\label{section 3}
 In this section we recall the criteria for generic finiteness and birationality of anti-pluri-canonical systems of canonical weak Fano $3$-folds from \cite{CJ16, Phd}.

 We refer to \cite{CJ16} for basic definitions.
 Here we should remind that in \cite{CJ16}, a $\bQ$-factorial terminal Fano $3$-fold with $\rho=1$ is called a {\it $\mathbb{Q}$-Fano $3$-fold}, and a terminal weak Fano $3$-fold is called a {\it weak $\mathbb{Q}$-Fano $3$-fold}. As a canonical weak Fano $3$-fold has a terminalization which is a crepant birational morphism from a terminal weak Fano $3$-fold, all results in \cite{CJ16}, which hold for terminal weak Fano $3$-folds, also hold for canonical weak Fano $3$-folds (see \cite[Remark~1.9]{CJ16}).

 Let $X$ be a canonical weak Fano $3$-fold and $m$ a positive integer such that $h^0(X, -mK_X)\geq
2$. We consider the rational map $\varphi_{-m}$ defined by the linear system $|-mK_X|$. We say that $|-mK_X|$ defines a generically finite map (resp. a birational map), if $\varphi_{-m}$ is 
 generically finite (resp. birational) onto its image. We say that $|-mK_X|$ is {\it composed with a pencil} if the image of $\varphi_{-m}$ is a curve. 
 
 \begin{set}\label{setup}
 \begin{enumerate}
 \item Let $X$ be a canonical weak Fano $3$-fold and $m_0$ a positive integer such that $h^0(X, -m_0K_X)\geq
2$. 

\item Suppose that $m_1\geq m_0$ is an integer with $h^0(X, -m_1K_X)\geq 2$ such that $|-m_1K_X|$ is not composed with a pencil.

\item Take a resolution $\pi: W\to X$ such that
$$
\pi^*(-mK_X)=M_{m}+F_{m}
$$
where $M_m$ is free and $F_{m}$ is the fixed part for all $m_0\leq m\leq m_1$. 

\item Pick a generic irreducible element $S$ of $|M_{m_0}|$ (a generic irreducible element means an irreducible component of a general element in \cite{CJ16}). We have $m_0\pi^*(-K_X)\geq S$.

\item 
Define the real number 
$$\mu_0:=\text{inf}\{t\in \bQ^+ \mid t\pi^*(-K_X)-S\sim_{\bQ} \text{effective}\ \bQ\text{-divisor}\}.$$
It is clear that $\mu_0\leq m_0$.

\item 
By the assumption on $|-m_1K_X|$, we know that $|{M_{m_1}}|_{S}|$ is a non-trivial base point free linear system. Denote by
$C$ a generic irreducible element of $|{M_{m_1}}|_{S}|$.

\item 
Define
\begin{align*}
\zeta{}&:=(\pi^*(-K_X)\cdot C);\\
\varepsilon(m){}&:=(m+1-\mu_0-m_1)\zeta.
\end{align*}
 \end{enumerate}

 \end{set}

Under this setup, we recall 2 propositions from \cite{CJ16, Phd}.

 \begin{prop}[{\cite[Proposition~5.7]{CJ16}}]\label{prop zeta}
Let $X$ be a canonical weak Fano $3$-fold. Keep the notation in Setup~\ref{setup}.
\begin{enumerate}
\item If $g(C)>0$, then $\zeta\geq \frac{2g(C)-1}{\mu_0+m_1}$.
\item If $g(C)=0$, then $\zeta\geq 2$.
\end{enumerate}
\end{prop}

 \begin{prop}\label{prop gen finite bir}
Let $X$ be a canonical weak Fano $3$-fold with $h^0(X, -K_X)>0$. Keep the notation in Setup~\ref{setup}. Take an integer $m\geq m_0+m_1+1$.
\begin{enumerate}
\item If $\varepsilon(m) > \max\{2-g(C), 0\}$, then $|-mK_X|$ defines a generically finite map.
 \item If $\varepsilon(m) > 2$, then $|-mK_X|$ defines a birational map.
\end{enumerate}
\end{prop}
\begin{proof}
This is by \cite[Theorem~5.5]{CJ16} and \cite[Theorem~4.4.4]{Phd}. Here we only need to explain that when $m\geq m_0+m_1+1$, \cite[Assumption~5.3]{CJ16} and \cite[Assumption~4.4.3]{Phd} hold. By the proof of \cite[Propositions~5.8 and 5.10]{CJ16}, these assumptions hold as long as $m\geq m_0+m_1+k_0$ where $k_0$ is an integer such that $h^0(X, -kK_X)>0$ for all $k\geq k_0$ (we take $k_0=6$ in \cite{CJ16, Phd}). As $h^0(X, -K_X)>0$, we can take $k_0=1$ in this proposition.
\end{proof}

By Propositions~\ref{prop zeta} and \ref{prop gen finite bir}, we get the following consequence on the behavior of anti-pluri-canonical systems.
 \begin{cor}\label{cor bir criterion}
Let $X$ be a canonical weak Fano $3$-fold with $h^0(X, -K_X)>0$. Keep the notation in Setup~\ref{setup}. 
\begin{enumerate}
\item If $m\geq 2m_0+2m_1$, then $|-mK_X|$ defines a generically finite map.
\item If $g(C)\neq 1$ and $m\geq m_0+m_1+1$, then $|-mK_X|$ defines a generically finite map.
 \item If $m\geq 3m_0+3m_1$, then $|-mK_X|$ defines a birational map.
\end{enumerate}
\end{cor}

 \section{Proofs of main theorems}

 In this section, we prove Theorems~\ref{mainthm} and \ref{mainthm2}.
 We will analyze the behavior of anti-pluri-canonical systems.
 \begin{lem}\label{lem 1}
Let $X$ be a canonical weak Fano $3$-fold such that $(-K_X)^3=(-K_{X_{6d}})^3$ and $B_{X}=B_{X_{6d}}$ for some $$X_{6d}\subset \mathbb{P}(1,a,b,2d, 3d)$$ as in Table~\ref{tableA}. Then

\begin{enumerate}
\item $h^0(X, -kK_X)>0$ for all positive integer $k$;
\item $h^0(X, -aK_X)\geq 2$.
\end{enumerate}
\end{lem}
\begin{proof}
This can be directly verified by Lemma~\ref{lem Xd RR}.
\end{proof}

 \begin{lem}\label{lem non-pencil}
Let $X$ be a $\bQ$-factorial terminal Fano $3$-fold with $\rho(X)=1$ such that $(-K_X)^3=(-K_{X_{6d}})^3$ and $B_{X}=B_{X_{6d}}$ for some $$X_{6d}\subset \mathbb{P}(1,a,b,2d, 3d)$$ as in Table~\ref{tableA}. Then $|-bK_X|$ is not composed with a pencil.
\end{lem}
\begin{proof}
If $b=1$, then $h^{0}(X, -K_X)=3$. Then $|-K_X|$ is not composed with a pencil by \cite[Theorem~3.2]{CJ16} (with $m=1$).

If $a=1<b$, then $h^0(X, -K_X)=2$. Then $|-K_X|$ has no fixed part by \cite[Theorem~3.2]{CJ16} (with $m=1$).
Moreover,
by Lemma~\ref{lem Xd RR}, $h^0(X, -kK_X)=k+1$ for $1\leq k \leq b-1$ and $h^0(X, -bK_X)=b+2$. Then $|-bK_X|$ is not composed with a pencil by \cite[Theorem~3.4]{CJ16} (with $m=n_0=1$ and $l_0=b$).

If $a>1$, then $h^0(X, -K_X)=1$. Then the unique element of $|-K_X|$ is a prime divisor by \cite[Theorem~3.2]{CJ16} (with $m=1$).
Moreover, 
by Lemma~\ref{lem Xd RR}, $$h^0(X, -kK_X)=\begin{cases}1 & \text{if }1\leq k\leq a-1; \\ \lfloor k/a\rfloor+1 & \text{if } a\leq k\leq b-1;\\\lfloor b/a\rfloor+2 & \text{if } k=b.\end{cases}$$
Then $|-bK_X|$ is not composed with a pencil by \cite[Theorem~3.4]{CJ16} (with $m=1$, $n_0=a$, and $l_0=b$).
\end{proof}

Lemma~\ref{lem non-pencil} is the only place that we use the condition that $X$ is $\mathbb{Q}$-factorial terminal with $\rho(X)=1$. Its proof relies heavily on \cite[Theorem~3.2]{CJ16} which generalizes \cite[Theorem~2.18]{Ale94}. Without this assumption, we have no idea how to show the non-pencilness of $|-bK_X|$ in Lemma~\ref{lem non-pencil}.
On the other hand, once we know the non-pencilness of $|-bK_X|$, then the geometry of $X$ is as good as we expect. So we make the following definition.

\begin{definition}
We say that a canonical weak Fano $3$-fold $X$ satisfies condition (IF$_{a, b}$) if the following conditions hold:
\begin{enumerate}
 \item $(-K_X)^3=(-K_{X_{6d}})^3$ and $B_{X}=B_{X_{6d}}$ for some $$X_{6d}\subset \mathbb{P}(1,a,b,2d, 3d)$$ as in Table~\ref{tableA}.
 \item $|-bK_X|$ is not composed with a pencil.
\end{enumerate}
\end{definition}

Then we can describe the behavior of anti-pluri-canonical systems of weak Fano $3$-folds satisfying condition (IF$_{a,b}$) which is the same as weighted hypersurfaces in Table~\ref{tableA} (compare \cite[Example~5.12]{CJ16}, \cite[Example~4.4.12]{Phd}).
\begin{lem}\label{lem gen finite}
Let $X$ be a canonical weak Fano $3$-fold satisfying condition (IF$_{a,b}$). Then 
\begin{enumerate}
 \item $|-2dK_X|$ defines a generically finite map;
 \item $|-3dK_X|$ defines a birational map.
\end{enumerate}
\end{lem}

\begin{proof}
This follow from Corollary~\ref{cor bir criterion} (taking $m_0=a$ and $m_1=b$). 
\end{proof}

Combining with Reid's Riemann--Roch formula, we get algebraic information of the anti-canonical rings of $X$ from the anti-canonical geometry of $X$.
\begin{lem}\label{lem fghp}
Let $X$ be a canonical weak Fano $3$-fold satisfying condition (IF$_{a,b}$). Take general elements \begin{align*}
 {}&f\in H^0(X, -K_X)\setminus\{0\},\\
 {}&g\in H^0(X, -aK_X)\setminus\{0\},\\
 {}&h\in H^0(X, -bK_X)\setminus\{0\}, \\
 {}&p\in H^0(X, -2dK_X)\setminus\{0\},\\
 {}&q\in H^0(X, -3dK_X)\setminus\{0\}.
\end{align*}
For a positive integer $k$, denote $$
S_k=\{f^{s_1}g^{s_2}h^{s_3}p^{s_4}\mid s_1, \dots, s_4\in \mathbb{Z}_{\geq 0}, s_1+as_2+bs_3+2ds_4=k\}.
$$
Then 
\begin{enumerate}
 \item The set $\{f, g, h, p\}$ is algebraically independent in the graded algebra $$R(X, -K_X)=\bigoplus_{k\geq 0}H^0(X, -kK_X);$$ 
 \item for $1\leq k\leq 3d-1$, $H^{0}(X, -kK_X)$ is spanned by a basis $S_k$;

\item $H^{0}(X, -3dK_X)$ is spanned by a basis $S_{3d}\cup\{q\}.$
\end{enumerate}
\end{lem}
\begin{proof}
First we show the following claim. 
\begin{claim}\label{claim fgh}
 The set $\{f, g, h\}$ is algebraically independent.
\end{claim}
\begin{proof}
If $a=b=1$, then as $h^0(X, -bK_X)=3$, $\{f, g, h\}$ is a basis of $H^0(X, -K_X)$. By assumption, $|-bK_X|=|-K_X|$ is not composed with a pencil, which is equivalent to say that the rational map $X\dashrightarrow \mathbb{P}^2$ induced by $f, g, h$ is dominant. 
This implies that $\{f, g, h\}$ is algebraically independent.

If $b>1$, we have $h^0(X, -aK_X)=2$. So $H^0(X, -aK_X)$ is spanned by $\{f^a, g\}$ and $\{f, g\}$ is algebraically independent. So $$S_b\setminus \{h\}=\{f^{b-as_2}g^{s_2}\mid 0\leq s_2\leq \lfloor b/a\rfloor\}$$ is linearly independent in $H^0(X, -bK_X)$. On the other hand, $h^0(X, -bK_X)=\lfloor b/a\rfloor+2=|S_b|$ by Lemma~\ref{lem Xd RR}. So $S_b$ is a basis of $H^0(X, -bK_X)$ by the generality of $h$.
By the assumption that $|-bK_X|$ is not composed with a pencil, the transcendental degree of $\mathbb{C}(f, g, h)$ is large than $2$, which implies that $\{f, g, h\}$ is algebraically independent.
\end{proof}
By Claim~\ref{claim fgh}, 
$$
S_k'=\{f^{s_1}g^{s_2}h^{s_3}\mid s_1, s_2, s_3\in \mathbb{Z}_{\geq 0}, s_1+as_2+bs_3=k\}
$$
is linearly independent in $H^0(X, -kK_X)$ for any positive integer $k$. 
On the other hand, $h^0(X, -2dK_X)=|S'_{2d}|+1$ by Lemma~\ref{lem Xd RR}. Then $S_{2d}=S'_{2d}\cup \{p\}$ is a basis of $H^0(X, -2dK_X)$ by the generality of $p$. By Lemma~\ref{lem gen finite}, $|-2dK_X|$ defines a generically finite map, so the transcendental degree of $\mathbb{C}(f, g, h, p)$ is large than $3$, which implies that $\{f, g, h, p\}$ is algebraically independent. This proves (1).

In particular, $S_k$ is linearly independent in $H^0(X, -kK_X)$ for any positive integer $k$. So (2) follows from the computation that $h^0(X, -kK_X)=|S_{k}|$ for $1\leq k\leq 3d-1$ by Lemma~\ref{lem Xd RR}, and (3) follows from the computation that $h^0(X, -3dK_X)=|S_{3d}|+1$ and the generality of $q$.
\end{proof}

The following is the key lemma of this note, which allows us to drop the non-rationality condition in \cite{330}.

\begin{lem}\label{lem deg=2}
Let $X$ be a canonical weak Fano $3$-fold satisfying condition (IF$_{a,b}$). 
Then
\begin{enumerate}
 \item $|-kK_X|$ does not define a generically finite map for $k<2d$;
 \item $|-2dK_X|$ defines a generically finite map of degree $2$.
 \item Let $\pi: W\to X$ be a resolution such that for $k=a, b, 2d$, $$\pi^*(-kK_X)=M_{k}+F_{k}$$ where $M_{k}$ is free and $F_{k}$ is the fixed part. Then $$(M_{a}\cdot M_{b}\cdot M_{2d})=2.$$
\end{enumerate}
\end{lem}

\begin{proof}
By Lemma~\ref{lem fghp}, $S_k$ has transcendental degree at most $3$ for $k<2d$. So $|-kK_X|$ does not define a generically finite map for $k<2d$.

Keep the same notation as in Setup~\ref{setup}. We may take $m_0=a$ and $m_1=b$. Then by construction, $S\in |M_a|$ and $C$ is a generic irreducible element of $|M_b|_S|$. If $g(C)\neq 1$, then by Corollary~\ref{cor bir criterion}, $|-(a+b+1)K_X|$ defines a generically finite map, which contradicts (1) as $a+b+1=d+1<2d$. So $g(C)=1$. 
If $|-2dK_X|$ defines a generically finite map of degree $d_0$, since $M_{2d}$ is the free part of $|\pi^*(-2dK_X)|$ and $C$ is general, $|M_{2d}|_C|$ defines a generically finite map of degree $d_0$ on $C$. 
On the other hand, $$(M_{2d}\cdot C)\leq (M_{2d}\cdot M_a\cdot M_b)\leq 2abd\pi^*(-K_X)^3 =2.$$
Note that a divisor of degree $1$ on an elliptic curve is never movable, so $M_{2d}\cdot C=2$.
Therefore, $(M_{2d}\cdot M_a\cdot M_b)=2$ and $d_0=2$.
\end{proof}

\begin{thm}\label{mainthm3}
Let $X$ be a canonical weak Fano $3$-fold satisfying condition (IF$_{a,b}$). 
Then $X$ is a weighted hypersurface in $\mathbb{P}(1,a,b,2d,3d)$ defined by a weighted homogeneous polynomial $F$ of degree $6d$, where 
$$F(x, y, z, w, t)=t^2+F_0(x, y, z, w)$$ in suitable homogeneous coordinates $[x: y: z: w: t]$ of $\mathbb{P}(1, a, b, 2d, 3d)$.
\end{thm}

 \begin{proof}
Keep the notation in Lemma~\ref{lem fghp}.
We can define $3$ rational maps by $\{f,g,h,p, q\}$:
\begin{align*}
\Phi_{b}: {}&X\dashrightarrow \mathbb{P}(1, a, b); \\
{}& P\mapsto [f(P):g(P):h(P)];\\
\Phi_{2d}: {}&X\dashrightarrow \mathbb{P}(1, a, b, 2d); \\
{}& P\mapsto [f(P):g(P):h(P):p(P)];\\
\Phi_{3d}: {}&X\dashrightarrow \mathbb{P}(1, a, b, 2d, 3d);\\
{}& P\mapsto [f(P):g(P):h(P):p(P): q(P)].
\end{align*}
We claim that they have the following geometric properties.
\begin{claim}Keep the above settings.
\begin{enumerate}

\item $\Phi_{b}$ is dominant; $\Phi_{2d}$ is dominant and generically finite of degree $2$;
\item $\Phi_{3d}$ is birational onto its image;
\item let $Y$ be the closure of $\Phi_{3d}(X)$ in $\mathbb{P}(1, a, b, 2d, 3d)$, then $Y$ is defined by a weighted homogeneous polynomial $F$ of degree $6d$, where 
$$F(x, y, z, w, t)=t^2+F_0(x, y, z, w)$$ in suitable homogeneous coordinates $[x: y: z: w: t]$ of $\mathbb{P}(1, a, b, 2d, 3d)$.
\end{enumerate}\end{claim}
\begin{proof}
(1) By Lemma~\ref{lem fghp}, $\{f, h, h, p\}$ is algebraically independent. Hence $\Phi_{b}$ and $\Phi_{2d}$ are dominant. In particular, $\Phi_{22}$ is generically finite by dimension reason.
The degree of $\Phi_{22}$ is the number of points in the fiber over a general point in $\mathbb{P}(1, a, b, 2d)$. After taking a resolution as in Lemma~\ref{lem deg=2}(3), this number is just $(M_a\cdot M_b\cdot M_{2d})$. So 
$$
\deg \Phi_{2d}=(M_a\cdot M_b\cdot M_{2d})=2.
$$
 
 \medskip
 
(2) By Lemma~\ref{lem gen finite}, $|-3dK_X|$ defines a birational map. As $q$ is general, it can separate two points in a general fiber of $\Phi_{2d}$, so $\Phi_{3d}$ is birational onto its image.

\medskip

(3) Note that $h^0(X, -6dK_X)=|S_{6d}|+|S_{3d}|$ by Lemma~\ref{lem Xd RR}. On the other hand, 
$$S_{6d}\sqcup (S_{3d}\cdot q)\sqcup \{q^2\}\subset H^0(X, -6dK_X).$$ 
So $S_{6d}\sqcup (S_{3d}\cdot q)\sqcup \{q^2\}$ is linearly dependent in $H^0(X, -6dK_X)$. In other words, there exists a weighted homogeneous polynomial $F(x, y, z, w, t)$ of degree $6d$ with $\text{wt}(x, y, z, w, t)=(1,a,b,2d,3d)$ such that 
$$F(f, g, h, p, q)=0.$$
So $Y$ is contained in $(F=0)\subset \mathbb{P}(1, a, b, 2d, 3d)$. Note that $Y$ is a hypersurface in $\mathbb{P}(1, a, b, 2d, 3d)$ by dimension reason.

We claim that $Y=(F=0)$ and $t^2$ has non-zero coefficient in $F$.
Otherwise, either $Y$ is defined by a weighted homogeneous polynomial of degree $\leq 65$, or $t^2$ has zero coefficient in $F$. In either case,
$Y$ is defined by a weighted homogeneous polynomial $\tilde{F}$ of the form 
$$\tilde{F}(x, y, z, w, t)=t\tilde{F}_1(x, y, z, w)+\tilde{F}_2(x, y, z, w).$$
Here $\tilde{F}_1\neq 0$, as $\{f, g, p, q\}$ is algebraically independent.
Then $Y$ is birational to $\mathbb{P}(1, a, b, 2d)$ under the rational projection map
\begin{align*}
{}&\mathbb{P}(1, a, b, 2d, 3d)\dashrightarrow \mathbb{P}(1, a, b, 2d);\\
{}&[x:y:z:w:t]\mapsto [x:y:z:w].
\end{align*}
But the induced map $X\dashrightarrow Y\dashrightarrow \mathbb{P}(1, a, b, 2d)$ coincides with $\Phi_{2d}$, which contradicts the fact that $\Phi_{2d}$ is not birational.


So $Y=(F=0)$ and $t^2$ has non-zero coefficient in $F$. After a suitable coordinate change we may assume that $F=t^2+F_0(x, y, z, w)$.
\end{proof}

Now go back to the proof of Theorem~\ref{mainthm3}.
By the above claim, $F$ is the only algebraic relation on $f, g, h, p, q$. Denote $\mathcal{R}$ to be the graded sub-$\mathbb{C}$-algebra of $$R(X, -K_X)=\bigoplus_{m\geq 0}H^0(X, -mK_X)$$
generated by $\{f, g, h, p, q\}$. Then we have a natural isomorphism between graded $\mathbb{C}$-algebras
$$
\mathcal{R}\simeq \mathbb{C}[x, y, z, w, t]/(t^2+F_0)
$$
by sending $f\mapsto x$, $g\mapsto y$, $h\mapsto z$, $p\mapsto w$, $q\mapsto t$ and the right hand side is exactly the weighted homogeneous coordinate ring of $Y$.
Write $\mathcal{R}=\bigoplus_{m\geq 0}\mathcal{R}_m$ where $\mathcal{R}_m$ is the homogeneous part of degree $m$. Then
by \cite[3.4.2]{Dol82}, 
$$
\sum_{m\geq 0}\dim_\mathbb{C} \mathcal{R}_m \cdot q^m= \frac{1-q^{6d}}{(1-q)(1-q^{a})(1-q^{b})(1-q^{2d})(1-q^{3d})}.
$$
So by Lemma~\ref{lem Xd RR}, $\mathcal{R}_m=H^0(X, -mK_X)$ for any $m\in \mathbb{Z}_{\geq 0}$, and hence the inclusion $\mathcal{R}\subset R(X, -K_X)$ is an isomorphism.
Since $-K_X$ is ample, this implies that 
$$X\simeq \Proj R(X, -K_X) = \Proj\mathcal{R}\simeq Y. $$
This finishes the proof.
 \end{proof}

 \begin{proof}[Proof of Theorem~\ref{mainthm2}]
It follows from Lemma~\ref{lem non-pencil} and Theorem~\ref{mainthm3}.
 \end{proof}

 \begin{proof}[Proof of Theorem~\ref{mainthm}]
 By \cite[Theorem~1.1]{CC08}, if $(-K_X)^3=\frac{1}{330}$, then $B_{X}=B_{X_{66}}$ for $$X_{66}\subset \mathbb{P}(1,5,6,22, 33)$$ as in Table~\ref{tableA}. Hence the theorem is a special case of Theorem~\ref{mainthm2}.
 \end{proof}

\section{Another approach via the Noether--Fano--Iskovskikh inequality and general elephants}

In this section, we discuss another possible approach to Theorem~\ref{mainthm2}. 

Keep the notation in Lemma~\ref{lem fghp}, following the proof in \cite{330} and Theorem~\ref{mainthm2}, the essential point is to show that the map \begin{align*}
\Phi_{2d}: {}&X\dashrightarrow \mathbb{P}(1, a, b, 2d); \\
{}& P\mapsto [f(P):g(P):h(P):p(P)]
\end{align*}
is not birational. 

Suppose that $\Phi_{2d}$ is birational. Note that $\mathcal{O}_{\mathbb{P}}(2abd)$ is very ample on $\mathbb{P}(1, a, b, 2d)$ and the strict transform of $|\mathcal{O}_{\mathbb{P}}(2abd)|$ on $X$ is the movable linear system $\mathcal{M}\subset |-2abdK_X|$ generated by $S_{2abd}\subset H^0(X, -2abdK_X)$. As $X$ is a $\mathbb{Q}$-factorial terminal Fano $3$-fold with $\rho(X)=1$, the Noether--Fano--Iskovskikh inequality (\cite{Cor95}, \cite{CPR00}) implies that $(X, \frac{1}{2abd}\mathcal{M})$ is not canonical. 

So in order to show that $\Phi_{2d}$ is not birational, it suffices to show that $(X, \frac{1}{2abd}\mathcal{M})$ is canonical. Furthermore, by Bertini type theorem, it suffices to show that $(X, E)$ is canonical, where $E\in |-K_X|$ is defined by $f\in H^0(X, -K_X)$, which is equivalent to say that a general element in $|-K_X|$ has at worst Du Val singularities. The latter one is called the {\it general elephant conjecture} proposed by Reid. If $X$ is a smooth Fano $3$-fold, then an earlier work of Shokurov \cite{Sho} showed that a general element of $|-K_X|$ is smooth. In general there are counterexamples to the {general elephant conjecture} for terminal Fano $3$-folds (see for example \cite[Examples~4.2--4.5]{Sano}), but we might still hope that it holds for Fano $3$-folds in Theorem~\ref{mainthm2}.
In fact, there is an interesting 1-1 correspondence between general elephants of weighted hypersurface Fano $3$-folds in Iano--Fletcher's list \cite[16.6]{IF00} and weighted hypersurface Du Val K3 surfaces in Reid's list \cite[13.3]{IF00}.
See also \cite{Ale94} for discussions on the general elephant conjecture.

At the last, recall that our goal is to show that $(X, \frac{1}{2abd}\mathcal{M})$ is canonical, which is weaker than the {general elephant conjecture}. One can try to prove this by methods in \cite{CP17}. The bad news is that we have no hypersurface structure on $X$ and it is not easy to find good divisors and curves as in \cite{CP17}; on the other hand, we know a lot on singularities of $X$ and the behavior of anti-pluri-canonical systems, which might help us to shape the geometry of $X$. For example, it might be possible to show that the movable linear system $\mathcal{M}$ is free (this is true by the conclusion of Theorem~\ref{mainthm2}, but we are looking for a different approach here), and hence $(X, \frac{1}{2abd}\mathcal{M})$ is automatically canonical by the Bertini theorem.

\section*{Acknowledgments}
The author was supported by National Key Research and Development Program of China (Grant No.~2020YFA0713200) and NSFC for Innovative Research Groups (Grant No. 12121001).

\end{document}